\documentclass[a4paper,reqno]{amsart}
\usepackage[footnotesize, bf]{caption}
\usepackage{a4wide}

\usepackage{setspace}
\usepackage[totalwidth=14cm,totalheight=22cm]{geometry}
\usepackage[T1]{fontenc}
\usepackage{tikz}
\usetikzlibrary{shapes,calc,cd,decorations.pathmorphing,hobby}

\usepackage{caption}
\usepackage{subfig}
\usepackage{color}
\usepackage{tikz-cd}
\usepackage{parcolumns} 
\usepackage{faktor}
\usepackage{marvosym}
\usepackage{listings}
\usepackage{booktabs}
\usepackage{adjustbox}
\usepackage{tabularx,ragged2e,booktabs}
\usepackage{hhline}
\usepackage{enumitem}
\usepackage{amsbsy}
\usepackage{amsfonts}
\usepackage{amsmath}
\usepackage{amsthm}
\usepackage{amssymb}
\usepackage{mathrsfs}
\usepackage{sidecap}
\usepackage{graphicx}
\usepackage{nicefrac}
\usepackage{wasysym}
\usepackage{relsize}
\usepackage{mathtools}
\usepackage{bbm}
\usepackage{braket}
\usepackage[english]{babel}
\usepackage{lmodern}
\usepackage[T1]{fontenc}
\usepackage{verbatim}
\usepackage{longtable}
\usepackage{microtype}
\usepackage{dsfont}
\usepackage{graphicx}
\usepackage{epstopdf}
\usepackage{fancyhdr}
\usepackage{tikz}
\usepackage{tikz-cd}
\usepackage{mathtools}
\usepackage{lipsum}
\usepackage[utf8]{inputenc}
\usepackage[T1]{fontenc}
\usepackage{lmodern} % load a font with all the characters 
\usepackage{xparse}
\usepackage{bm}
\usepackage{booktabs}
\usepackage{multirow}
\usepackage{imakeidx}

\usepackage{algorithm}
\usepackage[noend]{algpseudocode}

\usepackage{hyperref}
\hypersetup{linktocpage}

\newcommand{\numberset}{\mathbb}

\newcommand{\Z}{\numberset{Z}}
\newcommand{\R}{\numberset{R}}
\newcommand{\C}{\numberset{C}}

\newcommand{\blue}{\textcolor{blue}}

\newcommand{\GG}{ {\mathbb{G}}}

\newcommand{\JJ}{\mathds{J}}

\newcommand{\Hy}{\mathds{H}}

\newcommand{\phee}{\varphi}

\theoremstyle{definition}

\newtheorem{Definition}{Definition}[section]
\newtheorem*{Theorem*}{Theorem}

\newtheorem{Example}{Example}

\newtheorem{Remark}[Definition]{Remark}
\newtheorem{Proposition}[Definition]{Proposition}
\newtheorem{Prop}[Definition]{Proposition}

\newtheorem{Theorem}[Definition]{Theorem}

\theoremstyle{plain}

\newcommand{\PSL}{ {PSL(2,\C)} }
\newcommand{\CP}{\mathbb{C}\mathds{P}}

\newcommand{\SL}{ {SL(2,\C)} }
\newcommand{\asl}{\mathfrak{sl}(2,\C)}

\newcommand{\inner}[1] {\langle #1 \rangle}

\newcommand{\Proj}{\mathds{P}}
\newcommand{\GGG}{{\mathbb G}}

\newcommand{\inners}{\inner{\cdot,\cdot}}

\title{On $\PSL$ and on the space of geodesics of $\Hy^3$ as Riemannian holomorphic manifolds}
\author{CHRISTIAN EL EMAM}

\begin{document}
	\maketitle
	
	\begin{abstract}
	We discuss some geometric aspects of $\PSL$, $\SL$, and the space $\GG$ of the geodesics of $\Hy^3$ equipped with some suitable structures of Riemannian holomorphic manifolds of constant sectional curvature. We also observe that $\GG$ is a symmetric space for the group $\PSL$ and use it to deduce some correlations between their holomorphic Riemannian metrics.
	\end{abstract}

	\section*{Introduction}

	We examine some features of $\PSL$ and $\GG$, the space of geodesics of $\Hy^3$, equipped with some natural holomorphic Riemannian structures. The main reference for this paper is \cite{Articolone}.
	
	Holomorphic Riemannian metrics can be seen as the analogue of Riemannian metrics in the setting of complex manifolds and have been widely studied (e.g. see the works by LeBrun, Dumitrescu, Zeghib, Biswas as in  \cite{holomorphic riemannian 2}, \cite{holomorphic riemannian 3}, \cite{holomorphic riemannian 1}, \cite{holomorphic riemannian 4}). 
	
	We will recall the definition and some basic aspects of Holomorphic Remannian Geometry in Section \ref{Section HRM}.   
	 In Section \ref{Section PSL} we study $\PSL$ and $\SL$ equipped with their global Killing form as complex Lie groups, which is in fact a holomorphic Riemannian metric, while in Section \ref{Section G} we introduce a holomorphic Riemannian structure on $\GG$ over which the natural action of $\PSL$ is by isometries. 
	 
	 In Section \ref{HR space forms} we remark that both $\SL$ and $\GG$ can be seen as holomorphic Riemannian space forms, respectively in dimension three and two, and discuss some general topics of existence and uniqueness of holomorphic Riemannian space forms.
	
	Finally, in Section \ref{Section symm spaces} we regard $\Hy^3$ and $\GG$ as $\PSL$-symmetric spaces and show that this approach allows to find some correlations among their metrics.

	\section{Holomorphic Riemannian metrics}
	\label{Section HRM}
	Let us start by recalling the notion of holomorphic Riemannian metric and some generalities.

	Let $\mathbb M$ be a complex analytic manifold with complex structure $\JJ$, let $n=dim_\C \mathbb M$ and denote with $T\mathbb M\to \mathbb M$ its tangent bundle.
	
	\begin{Definition}
		A \emph{holomorphic Riemannian metric} (also \emph{hRm}) on $\mathbb M$ is a symmetric $(0,2)$-tensor $\inner{\cdot, \cdot}$ on $T\mathbb M\to \mathbb M$, i.e. a section of $Sym^2 (T^*\mathbb M)$, such that: 
		\begin{itemize}
			\item $\langle \cdot, \cdot \rangle$ is $\C$-bilinear, i.e. for all $X,Y \in \Gamma (T\mathbb M)$ we have $\langle \JJ X, Y \rangle= \langle X, \JJ Y\rangle =i \langle X, Y\rangle$;
			\item for all $p\in \mathbb M$, $\inners_p$ is a non-degenerate complex bilinear form;
			\item for all $X, Y$ local holomorphic vector fields, $\inner{X, Y}$ is a holomorphic function.
		\end{itemize}
	Also denote $\|X \|^2:= \inner{X,X}$.
	
	A \emph{holomorphic Riemannian manifold} is a complex manifold equipped with a holomorphic Riemannian metric. 
	\end{Definition}
	
	In a sense, holomorphic Riemannian metrics naturally generalize the notions of Riemannian and pseudo-Riemannian metric in the complex setting. 
	
	Observe that both the real part $Re\inner{\cdot, \cdot}$ and the imaginary part $Im\inner{\cdot, \cdot}$ of a hRm $\inners$  are pseudo-Riemannian metrics on $\mathbb M$ with signature $(n,n)$.

	\begin{Example}
		\label{Esempio1}
		 The simplest example of hRm is given by $\C^n$ equipped with the standard symmetric inner product for vectors: namely, in the usual identification $T\C^n =\C^n\times \C^n$, 
	\[\inner{\underline v, \underline w}_{\C^n}= \sum_{i=1}^n v_iw_i.
	\]
\end{Example}
\begin{Example}
	\label{Esempio2}
Consider the complex submanifold
\[\mathds{X}_n=\{\underline z\in \C^{n+1}\ |\ \sum_i z_i^2=-1 \}\subset \C^{n+1}.\]

The restriction to $\mathds X_n$ of the metric $\inners_0$ of $\C^{n+1}$ defines a holomorphic Riemannian metric. Indeed,
\[
T_{\underline w} \mathds{X}_n = {\underline w}^\bot= \{\underline z\in \C^{n+1}\ |\ \inner{\underline w,\underline z}_{\C^{n+1}}=0\} 
\]
and the restriction of the inner product to $\underline w^\bot$ is non degenerate since $<\underline w,\underline w>_{\C^{n+1}}\ne 0$; moreover, since $\mathds X_n\subset \C^{n+1}$ is a complex submanifold, local holomorphic vector fields on $\mathds X_n$ extend to local holomorphic vector fields on $\C^{n+1}$, proving that the inherited metric is in fact holomorphic.
\end{Example}
\begin{Example}
	\label{Esempio3}
 Let $G$ be a complex semisimple Lie group with unit $e$. 

Consider on $T_e G\cong Lie(G)$ the \emph{Killing form} $Kill\colon T_e G \times T_e G \to \C$, defined as $Kill(u,v):= tr(ad(u)\circ ad(v) )$. 

By standard Lie Theory, the Killing form is $\C$-bilinear and symmetric; $Kill$ is also \emph{$Ad$-invariant}, i.e. for all $g\in G$ \[Kill(Ad(g) \cdot, Ad(g) \cdot)=Kill\], and \emph{$ad$-invariant}, i.e. for all $V\in Lie(G)$ \[Kill(ad(V)\cdot, \cdot)+ Kill(\cdot, ad(V)\cdot)=0.\] Furthermore, being $G$ semisimple, $Kill$ is non degenerate. 

For all $g\in G$ one can push-forward $Kill$ via $L_g$ to define a non degenerate $\C$-bilinear form on $T_g G$, namely 
\[Kill_g (U,V):= \Big((L_g)_* Kill \Big)(U,V)=  Kill\bigg( \big(d_g (L_g^{-1})\big)(U), \big(d_g (L_g^{-1})\big)(V) \bigg)\]
for all $U,V\in T_g G$.
By $Ad$-invariance, the analogous bilinear form $(R_g)_* Kill$ is such that $(R_g)_* Kill=(L_g)_* Kill$.

This defines globally a nowhere-degenerate section $Kill_{\bullet} \in \Gamma\big(Sym^2(T^*G)\big)$ such that, for all $X,Y$ left-invariant vector fields, hence holomorphic vector fields (e.g. see \cite{complex Lie groups} \S1), $Kill_{\bullet}(X,Y)$ is constant, hence a holomorphic function. Since any other holomorphic vector field can be seen as a combination of left-invariant vector fields with holomorphic coefficients, we conclude that $Kill_{\bullet}$ is a holomorphic Riemannian metric.
\end{Example}

\vspace{3mm}

	Drawing inspiration from (Pseudo-)Riemannian Geometry, one can define some constructions associated to a holomorphic Riemannian metric, such as a Levi-Civita connection - leading to notions of a curvature tensor, (complex) geodesics and completeness - and sectional curvatures. We recall some generalities, the reader may refer to \cite{holomorphic riemannian 4} for a more detailed treatment.

	\begin{Proposition}
		Given a holomorphic Riemannian metric $\inner{\cdot, \cdot}$ on $\mathbb M$, there exists a unique connection 
	\begin{equation*}
	\begin{split}
	D \colon \Gamma (T\mathbb M) &\to \Gamma(Hom_{\R} (T\mathbb M, T\mathbb M))\\
	Y &\mapsto D Y (\colon X \mapsto D_X Y),
	\end{split}
	\end{equation*}
		that we will call \emph{Levi-Civita connection}, such that for all $X,Y \in \Gamma(TM)$ the following conditions hold:
		\begin{align}
		\label{compatibilita con metrica}
		d \inner{X, Y} &= \inner{D X, Y}+ \inner{X, D Y} \qquad && \text{($D$ is compatible with the metric)};\\
		\label{torsion free}
		[X,Y]&= D_X Y - D_Y X \qquad &&\text{($D$ is torsion free)}.
		\end{align}
		Such connection coincides with the Levi-Civita connections of $Re\inner{\cdot, \cdot}$ and $Im\inner{\cdot, \cdot}$ and $D\JJ=0$.
	\end{Proposition}
	
	\begin{proof} The proof is definitely analogous to the one in the Riemannian setting. See \cite{holomorphic riemannian 4}. Explicitly the Levi-Civita connection is defined by 
			\begin{equation}
		\label{Levi-Civita}
		\begin{split}
		\langle D_X Y, Z\rangle= \frac 1 2 \Big( X\inner{Y,Z} + Y\inner{Z,X} - Z\inner{X, Y} + \\
		+\inner{[X,Y],Z} -\inner{[Y,Z],X}+ \inner{[Z,X],Y} \Big).
		\end{split}
		\end{equation}
	\end{proof}

	\begin{Definition}
		Given a holomorphic Riemannian metric $\langle\cdot, \cdot \rangle$ and its Levi-Civita connection $D$ on $T\mathbb M$, define the $(1,3)$-type \emph{curvature tensor} as 
		\[R(X,Y)Z=D_X D_Y Z- D_Y D_X Z - D_{[X,Y]} Z\]
		and the corresponding $(0,4)$-type curvature tensor (that we will still denote with $R$) as
		\[R(X,Y,Z,T)= -\inner{R(X,Y,Z), T}.\]
	\end{Definition}
	
	\begin{Remark}
		Since $D$ is the Levi-Civita connection for $Re\inner{\cdot, \cdot}$ and for $Im\inner{\cdot, \cdot}$, it is easy to check that all of the standard  symmetries of curvature tensors for (the Levi-Civita connections of) pseudo-Riemannian metrics hold for (the Levi-Civita connections of) holomorphic Riemannian metrics, too. So, for instance, 
		\[
		R(X,Y,Z,T)= -R(X, Y, T, Z)= R(Z, T, X, Y)= -R(Z, T, Y, X).
		\]
		Since the $(0,4)$-type $R$ is obviously $\C$-linear on the last component, we conclude that $R$ is $\C$-multilinear.
	\end{Remark}

	\iffalse
	\begin{Remark}	
		\label{curvature tensor Killing}
		Consider the Levi-Civita connection $D$ for the Killing form for a complex semisimple Lie group $G$. Let $[\cdot,\cdot]$ define both the Lie bracket on $T_e G$ and the commutator for vector fields. Then, extending the vectors $V_1, V_2, V_3\in \asl$ respectively to left-invariant vector fields $X, Y, Z$ on $G$, we get 
		\begin{align*}
		R(V_1, V_2) V_3 :=&( D_X D_Y Z - D_Y D_X Z - D_{[X,Y]} Z)_{|_e}=\\
		=& \frac 1 4 \Big( [X, [Y,Z]]- [Y, [X,Z]]\Big)_{|_e}- \frac 1 2 [[X,Y],Z]_{|_e}\\
		=&\frac 1 4 \Big( [X, [Y,Z]]+[Y,[Z,X]] + [Z,[X,Y]] \Big)_{|_e} - \frac 1 4 [[X,Y],Z]_{|_e}=\\
		=& -\frac 1 4 [[V_1,V_2],V_3].
		\end{align*}
	\end{Remark}
	\fi

	\begin{Definition}	A \emph{non-degenerate plane} of $T_p \mathbb M$ is a complex vector subspace $\mathcal V< T_p \mathbb M$ with $dim_\C\mathcal V=2$ and such that $\inner{\cdot, \cdot}_{|\mathcal V}$ is a non degenerate bilinear form.
		
		Given a hRm $\inners$ with curvature $R$, we define the {complex sectional curvature} of a nondegenerate complex plane $\mathcal V= Span_\C (V,W)$ as
		\begin{equation}
		\label{def curvatura}
		K(Span_\C (V,W))=\frac{\inner{R(V,W)W, V}}{\|V\|^2 \|W\|^2 - \inner{V,W}^2}\newline.
		\end{equation}
		This definition of $K(Span_\C (V,W))$ is well-posed because $\inners_{|\mathcal V}$ is non-degenerate and $R$ is $\C$-multilinear.
	\end{Definition}

\section{The holomorphic Riemannian manifold $(\PSL, \frac 1 8 Kill)$}
\label{Section PSL}
	The complex projective special linear group 
	\[
	\PSL=\faktor{ \{A\in Mat(2,\C)\ |\ det(A)=1\} }{\{\pm I_2\}}
	\]
	has a natural structure of complex semisimple Lie group with Lie algebra
	\[\asl=\{M\in Mat(2,\C)\ |\ tr(M)=0\}=T_{I_2}\PSL.
	\] 
	The complex Killing form of $\asl$ is given by
	\[
	Kill(M,N)= 4 tr(M \cdot N)
	\]
	and is in fact a symmetric, non-degenerate, $\C$-bilinear form.
	
As in Example $\ref{Esempio3}$, by pushing forward the Killing form via translations, one gets a global holomorphic Riemannian metric. As a matter of convenience we rescale it and define $\inners=\inners_\PSL \in \Gamma(Sym_2(T^*\PSL) )$ as
\[
\inners_A := \frac 1 8  (L_A)_* Kill= \frac 1 8  (R_A)_* Kill.
\]

The projection $\SL \to \PSL$ is a holomorphic $2$-sheeted covering, in fact a universal covering, and a homomorphism of Lie groups. As a result, the pull-back metric on $\SL$, that we will denote by $\inners_\SL$ or again simply by $\inners$, is a holomorphic Riemannian metric and coincides with the global Killing form on $\SL$ up to a factor.

\begin{Remark}
	\label{Remark automorfismi PSL}
	By construction of $\inners_\PSL$,  $\PSL\times \PSL$ acts by isometries on $\PSL$ through
	\begin{equation}
	\label{automorfismi SL}
	\begin{split}
	{\PSL \times \PSL}\to Isom(\PSL)\\
	(A,B)\cdot C := A\ C\ B^{-1}.
	\end{split}
	\end{equation}  
 Since the center of $\PSL$ is trivial, the action is faithful, hence the map $\eqref{automorfismi SL}$ is injective. 
	
	In fact, one has
	\[
	 \PSL \times \PSL\cong Isom_0(\PSL, \inners).
	\]
	In order to prove it, it is enough to show that the two Lie groups have the same dimension, namely, by injectivity of $(\ref{automorfismi SL})$, that $dim\big(Isom_0(\PSL, \inners)\big)\ge 6.$
	
	Since every isometry is an isomorphism for the Levi-Civita connection, for any $A\in \PSL$ and $\phee \in Isom(\PSL)$ one has that $\phee \circ \exp_A = \exp_{\phee(A)} \circ \phee$, hence, $\PSL$ being connected, the differential map at one point characterizes the isometry: as a result, denoting $Stab(I_2)=\{\phee\in Isom(\PSL)\ |\ \phee(I_2)=I_2 \}$, the map
	\begin{equation*}
	\begin{split}
	Stab(I_2)&\to Isom( T_{I_2} \PSL)\cong O(3,\C)\\
	\phee &\mapsto  d_{I_2}\phee
	\end{split}
	\end{equation*}
	is injective. In conclusion, \[dim_\C(Isom(\PSL))\le dim_\C(O(3,\C)) + dim_\C(\PSL)=6. \] As a result, $dim_\C Isom(\PSL)=6$ and the proof follows.
	
	In a similar way one gets that 
	\[
	Isom_0 (\SL, \inners_\SL)\cong \faktor{\SL \times \SL}{\{\pm(I_2, I_2)\}}
	\]
\end{Remark}

\begin{Proposition}
	\label{Prop PSL}
	Let $D$ be the Levi-Civita connection for $\inners=\inners_\PSL$ and $R$ the associated curvature tensor. 
	\begin{enumerate}
\item If $X$,$Y$ are left-invariant vector fields, then $D_X Y =\frac 1 2 [X,Y]$.
\item For all $V_1 V_2, V_3\in \asl$, $R(V_1,V_2)V_3=-\frac 1 4 [[V_1,V_2],V_3]$.
\item For all $V_1,V_2\in \asl$, $\| [V_1, V_2] \|^2 =-4 \|V_1\|^2 \|V_2\|^2 + 4 \inner{V_1,V_2}^2$.
\item $\PSL$ and $\SL$ have constant sectional curvature $-1$.
\end{enumerate}
\end{Proposition}

\begin{proof}
\begin{enumerate}
\item	Let $X, Y, Z$ be left-invariant vector fields for $\PSL$.		
		Then $\inner{X,Y},\inner{Y,Z}$ and $\inner{X,Z}$ are constant functions and $\inner{[Z,Y],X}+\inner{[Z,X],Y}=0$ since the Killing form is $ad$-invariant. In conclusion, by the explicit expression (\ref{Levi-Civita}), we get that 
		\[
		D_X Y =\frac 1 2 [X,Y]
		\]
		for all $X,Y$ left-invariant vector fields. An analogous statement holds for the complex global Killing form of any semi-simple complex Lie group.
\item 
Let $X, Y, Z$ be left-invariant vector fields and let $X_0,Y_0,Z_0\in \asl$ be their value at $I_2$ respectively. Then,
\begin{equation*}
\begin{split}
R(X_{0},Y_{0}) Z_{0} :=&( D_X D_Y Z - D_Y D_X Z - D_{[X,Y]} Z)_{|_{I_2}}=\\
=& \frac 1 4 \Big( [X, [Y,Z]]- [Y, [X,Z]]\Big)_{|_{I_2}}- \frac 1 2 [[X,Y],Z]_{|_{I_2}}\\
=&\frac 1 4 \Big( [X, [Y,Z]]+[Y,[Z,X]] + [Z,[X,Y]] \Big)_{|_{I_2}} - \frac 1 4 [[X,Y],Z]_{|_{I_2}}=\\
=& -\frac 1 4 [[X_{0},Y_{0}],Z_{0}]
		\end{split}
		\end{equation*}
		where the last Lie bracket is the Lie bracket of $\asl$.
\item The expression is $\C$-bilinear on $V$ and $V_2$, so it suffices to show that it holds for orthonormal $V_1$ and $V_2$. By Remark \ref{Remark automorfismi PSL}, the main connected component of the isometries of $\PSL$ that fix $I_2$ is isomorphic to $SO(3,\C)$, so $Stab(I_2)$ acts transitively on (unordered) couples of orthonormal vectors of $(\asl, \inners)$: as a result, one can check the equation on a particular couple of orthonormal vectors of $\PSL$:\[\bigg\| \begin{pmatrix}1 & 0\\
	0 & -1
\end{pmatrix}\bigg\|^2= \bigg\| \begin{pmatrix}0 & 1\\
	1 & 0
\end{pmatrix} \bigg \|^2 =1\] and
 \[\bigg[
\begin{pmatrix}1 & 0\\
0 & -1
\end{pmatrix}, \begin{pmatrix}0 & 1\\
1 & 0
\end{pmatrix}
\bigg] = \begin{pmatrix}0 & 2\\
-2 & 0
\end{pmatrix}\] which has squared norm $-4$.
\item 
Let $V_1, V_2\in \asl$ be orthonormal with respect to $\inners$. Using the previous steps and that the Killing form is $ad$-invariant, one has 
\[
K(Span(V_1,V_2))=\frac 1 4 \inner{[[V_1, V_2], V_1], V_2}=+ \frac 1 4 \inner{[V_1,V_2], [V_1,V_2]}=-1.
\]
The proof follows by homogeneity of $\inners_\PSL$ and $\inners_\SL$ 

	\end{enumerate}
\end{proof}

\section{The space $\GG$ of the geodesics of $\Hy^3$}
\label{Section G}
Let $\partial \Hy^n$ denote the visual boundary at infinity of $\Hy^n$.
For $n=3$, $\partial \Hy^3\cong S^2$ has a natural complex structure for which the trace at infinity of any element in $Isom_0(\PSL)$ is a biholomorphism of the boundary: for $\Hy^3$ in the half-space model, $\partial \Hy^3$ is biholomorphic to $\overline \C$, and the canonical identification $\PSL\cong Isom_0(\Hy^3)$ induces an isomorphism at infinity $\PSL\cong Bihol(\overline \C)$ defined by M\"obius transformations. We will equivalently see $\partial \Hy^3$ as $\overline \C$, $S^2$ or $\CP^1$ with the standard complex structures.

Define an \emph{oriented} (resp \emph{unoriented}) \emph{line} of $\Hy^3$ as a non-constant, maximal, oriented (resp. unoriented), non-parametrized geodesic of $\Hy^3$. Each oriented line of $\Hy^3$ is uniquely identified by the ordered pair of its endpoints at infinity, namely a ordered pair of distinct points in $\partial \Hy^3$. Hence, define
\[
\GG= \{ \text{oriented lines of } \Hy^3 \}= \partial \Hy^3\times \partial \Hy^3 \setminus \Delta.
\]
With a little abuse, we will often refer to $\GG$ as the set of geodesics of $\Hy^3$.

The set $\GG$ inherits a structure of complex manifold from $\partial \Hy^3$. Moreover, the action of $\PSL$ on $\Hy^3$ clearly induces an action on $\GG$: as a matter of fact, it is the diagonal action on $\partial \Hy^3$, namely
\begin{equation}
\begin{split}\PSL\times \GG &\to \GG\\
A \cdot (z_1,z_2)&= (Az_1, Az_2),
\end{split}
\end{equation}
which is clearly an action by biholomorphisms. 

\begin{Proposition}
	There exists a unique holomorphic Riemannian metric $\inners_\GG$ on $\GG$ with constant curvature $-1$ and with the property of being $\PSL$-invariant, more precisely $Isom_0(\GG, \inners_\GG)\cong \PSL$.
	
	Explicitly, if $(U,z)$ is an affine chart for $\CP^1$, $\inners_\GG$ is described in the coordinate chart $(U\times U\setminus \Delta, z\times z)$ as 
	\[\inners_\GG= -\frac{4}{(z_1-z_2)^2} dz_1 dz_2.\]
\end{Proposition}
\begin{proof}
	See \cite{Articolone}\S2.
\end{proof}

\subsection*{An interesting isometric immersion of $\GG$ in $\PSL$}
There exists a very intuitive isometric immersion of $\GG$ into $\PSL$.

Consider the complex manifold
\begin{equation}
\begin{split}
Q &= \PSL \cap \Proj (\asl) =\{ M\in \Proj Mat(2,\C) \ |\ tr(M)=0, det(M)=1 \}\subset \\
&\subset \Proj (Mat(2,\C))\cong \CP^3
.
\end{split}
\end{equation}
It is not difficult to remark that
\[
Q=\{M\in \PSL \ |\ M^2 =I_2 \text{ and } M\ne I_2 \}
\] 
via a straighforward computation.

Hence, in the identification $\PSL\cong Isom_0(\PSL)$, $Q$ corresponds to orientation-preserving isometries of $\Hy^3$ of order $2$, namely rotations of angle $\pi$ around an unoriented line of $\Hy^3$.

As a result, one can define a map
\[
Rot_\pi \colon \GG \to Q
\]
that sends $\gamma \in \GG$ into the rotation of angle $\pi$ around $\gamma$: $Rot_\pi$ is clearly a $2$-sheeted covering map. An explicit computation shows that $Rot_\pi$ is holomorphic and a local isometry between the hRm on $\GG$ and the hRm induced on $Q$ as a complex submanifold of $\PSL$ (also see \cite{Articolone}\S2).

\section{$\GG$ and $\PSL$ as holomorphic Riemannian space forms}
\label{HR space forms}

We say that a holomorphic Riemannian manifold $(\mathbb M, \inners)$ is a \emph{holomorphic Riemannian space form} if it has constant sectional curvature, is simply-connected and its Levi-Civita connection is geodesically complete.

Similarly as in the Riemannian case, we have a result about existence and uniqueness of space forms.

	\begin{Theorem}
	\label{Theorem space forms}
	For all $n\in \Z_+$, $n\ge 2$, and $k\in \C$, there exists exactly one holomorphic Riemannian space form of dimension $n$ with constant sectional curvature $k$ up to isometry, namely:
	\begin{itemize}
		\item $(\C^n,\inners_0)$, as in Example \ref{Esempio1}, for $k=0$;
		\item $(\mathds{X}_n, -\frac {1} k\inner{\cdot, \cdot})$, with $\inners$ as in Example \ref{Esempio2}, for $k \in \C^*$.
	\end{itemize}
\end{Theorem}
\begin{proof}
	See \cite{Articolone}\S2. 
\end{proof}

As a matter of fact, $(\mathds X_3, \inners)\cong (\SL, \inners)$ and $(\mathds X_2, \inners)\cong (\GG, \inners)$. 
One may explicitly check that $\SL$ and $\GG$ are geodesically complete and deduce the existence of the isometries by Theorem \ref{Theorem space forms}, but let us show that one can actually see the isometries directly.
\begin{itemize}
	\item
	Consider on $Mat(2,\C)$ the non-degenerate quadratic form given by $M\mapsto -det(M)$, which corresponds to the complex bilinear form 
	\[
	\inner{M,N}_{Mat_2}=\frac{1}{2}\bigg(tr(M\cdot N)- tr(M)\cdot tr(N) \bigg).
	\]
	In the identification $T Mat(2,\C)=Mat(2,\C)\times Mat(2,\C)$, this complex bilinear form induces a holomorphic Riemannian metric on $Mat(2,\C)$. Observe that the action of $\SL\times \SL$ on $Mat(2,\C)$ given by $(A,B)\cdot M:= A 
	M B^{-1}$ is by isometries, because it preserves the quadratic form.
	
	Since all the non-degenerate complex bilinear forms on complex vector spaces of the same dimension are isomorphic, there exists a linear isomorphism $F\colon (\C^4,\inners_{\C^4}) \to (Mat(2,\C),\inners_{Mat_2})$, which is of course an isometry of holomorphic Riemannian manifolds: such isometry $F$ restricts to an isometry between $\mathds X_3$ and $SL(2,\C)$ and the restriction of $\inners_{Mat_2}$ to $T_{I_2} \SL= \asl$ coincides by construction with $\frac 1 8 Kill$. 
	By invariance of $\inners_{Mat_2}$ under the action of $\SL \times \SL$, we conclude that the induced metric on $\SL$ coincides with the one defined previously.

	\item By choosing explicitly $F$ as	
	\begin{equation}
	\label{iso C4 Mat(2,C)}
	\begin{split}
	F\colon\C^4 &\to Mat(2,\C)\\
	(z_1, z_2, z_3, z_4)&\mapsto 
	\begin{pmatrix}
	-z_1- iz_4 &  -z_2-iz_3\\
	-z_2 + iz_3 & z_1-iz_4
	\end{pmatrix},
	\end{split}
	\end{equation}
	one can straightforwardly see that the image via $F$ of $\{(z_1, z_2, z_3, 0)\ | \sum_k z_k^2=-1 \}\cong \mathds X_2$ is $\SL \cap \asl$.
	
	Recalling the definition $Q=\PSL \cap \Proj(\asl)$, we have that 
	 both the projection $\SL \cap \asl \to Q$ and $Rot_\pi\colon \GG\to Q$ are isometric universal coverings for $Q$, hence $\mathds X_2\cong \GG$.
\end{itemize}

\section{$\Hy^3$ and $\GG$ as symmetric spaces of $\PSL$}
\label{Section symm spaces}
As we remarked previously, $
Isom_0(\Hy^3)\cong Isom_0(\GG)\cong \PSL.$

In fact, both $\Hy^3$ and $\GG$ can be seen as symmetric spaces (in the sense of affine spaces) associated to $\PSL$:
\begin{itemize}
	\item[a)]  $\Hy^3 \cong \faktor {\PSL}{SU(2)}$, where the symmetry at $0\in \Hy^3\subset \R^3$ in the disk model is given by the map $\underline x\mapsto -\underline x$;\\
	\item [b)]  $\GG=\faktor {\PSL}{SO(2,\C)}$, where the symmetry at a geodesic $\gamma\in \GG$ is given by the rotation of angle $\pi$ around $\gamma$, which is indeed an element of $Isom_0(\GG)$.
\end{itemize}
For a complete survey on symmetric spaces see \cite{Kobayashi-Nomizu 2}.

Our aim in this section is to show that the metrics on $\Hy^3$ and $\GG$ are in fact related to the hRm on $\PSL$ through their structures of symmetric spaces.

\vspace{5mm}

Fix $x_1 \in \Hy^3$ and $x_2 \in \GG$ and, as a matter of convenience, denote $X_1:= \Hy^3$ and $X_2:=\GG$.

For $k=1,2$, define the evaluation map related to $(X_k, x_k)$ as
\begin{equation}
\begin{split}
\beta_k\colon \PSL &\to X_k\\
A &\mapsto A\cdot x_k.
\end{split}
\end{equation}

The two marked symmetric spaces $(X_1,x_1)$ and $(X_2,x_2)$ induce two Cartan decompositions of the Lie algebra $\asl$, namely
\begin{equation}\begin{split}
\text{for $(X_1, x_1)$: } &\quad \asl= \mathfrak u (2) \oplus i \mathfrak u(2)=: \mathfrak h_1 \oplus \mathfrak m_1; \\
\text{for $(X_2, x_2)$: } &\quad \asl= \mathfrak o(2,\C) \oplus (Sym(2,\C)\cap \asl)=: \mathfrak h_2 \oplus \mathfrak m_2.
\end{split}
\end{equation}
We recall the following facts (see \cite{Kobayashi-Nomizu 2} \S X-XI):
\begin{itemize}
	\item $\mathfrak h_i =Lie(Stab(x_i))$, and the $Ad$ action of $Stab(x_i)$ on $\asl$ globally fixes $\mathfrak m_i$;
	\item $[\mathfrak h_i, \mathfrak m_i]\subset \mathfrak m_i$ and $[\mathfrak m_i, \mathfrak m_i]\subset \mathfrak h_i$;
\item the map 
\[
d_{I_2}\beta_i\colon \mathfrak m_i \to T_{x_i}  X_i
\]
is a linear isomorphism. Let us denote $d_{I_2}\beta_i (V)=: V_{x_i}$ for all $V\in \mathfrak m_i$;

\item For all $A\in \PSL$, one has \begin{equation}
\label{eq 1}
\beta_i \circ L_A= A\circ \beta_i.\end{equation}
As a corollary, for all $M\in Stab(x_i)$,\begin{equation}
\label{eq 2}
d_{I_2}\beta_i\circ Ad(A)= d_{x_i} A \circ d_{I_2} \beta_i.
\end{equation}

\item 	\label{Prop curvatura spazio simm}
For all $U, V, W\in \mathfrak m_i$, 
\begin{equation}
\label{eq curvatura spazio simmetrico}
R^{X_i}(U_{x_i}, V_{x_i})W_{x_i}= - [[U,V], W]
\end{equation}
where $R^{X_i}$ is the curvature tensor of $X_i$.
\end{itemize}

Define on $T\PSL$ the distribution $(\mathcal D_i)(A) := (L_A)_* \mathfrak m_i < T_A \PSL$.
\begin{Prop}
	For both $i=1,2$, the restriction to $\mathcal D_i$ of the differential of $\beta_i$ is a linear isometry at each $A\in \PSL$ up to a constant. Namely, for all $A\in \PSL$, 
	\[
	d_A \beta_i \colon (\mathcal D_i(A), 4 \inners) \xrightarrow{\sim} T_{A \cdot x_i} X_i.
	\]
\end{Prop}
\begin{proof}
	Both the Riemannian metric of $\Hy^3$ and the hRm of $\GG$ are uniquely determined by being $\PSL$-invariant metrics (Riemannian and holomophic Riemannian resp.) with constant sectional curvature $-1$. It is therefore enough to show that push-forward bilinear forms $(\beta_{i |\mathcal D_i})_* (\inners)$ define two well-posed, $\PSL$-invariant metrics of constant sectional curvature $-4$.
	
	By standard Lie theory, ${Kill^{\asl}}_{|\mathfrak u(2)}= Kill^{\mathfrak u(2)}$ which is a real negative-definite bilinear form being $U(2)$ compact and semisimple. By $\C$-bilinearity of $\inners=\frac 1 8 Kill^{\asl}$,  $\inners_{|i \mathfrak u(2)}$ is real and positive-definite. 
	
	On the other hand, $\inners_{|\mathfrak m_i}$ is a non-degenerate $\C$-bilinear form with orthonormal basis $\bigg( \begin{pmatrix} 0 & 1\\
	1 &0
	\end{pmatrix},  \begin{pmatrix} 0 & -i\\
	i & 0
	\end{pmatrix} \bigg)$.
	
	As we mentioned, consider on each $X_i$ the metric $g_i$ defined so that for all $A\in \PSL$ and for all $V,W\in \mathfrak m_i$
	\[
	(g_i)_{A\cdot x_i} \bigg((d_{x_i}A) (V_{x_i}), (d_{x_i} A) (W_{x_i})\bigg):= \inner{V, W}_{\asl}.
	\]
	\begin{itemize}
\item The definition of $(g_i)$ is well-posed. Indeed, if $A\cdot x_i=B\cdot x_i$, $(d_{x_i}A) V_{x_i}= (d_{x_i}B) V'_{x_i}$ and $(d_{x_i}A) W_{x_i}=(d_{x_i}B) W'_{x_i}$, then using that $\inners$ is $Ad$-invariant and equation \eqref{eq 2} one has
	\begin{align*}
	(g_i)_{A\cdot x_i} ( (d_{x_i}A) V_{x_i}, (d_{x_i}A) W_{x_i}):=& g_{x_i}( V_{x_i}, W_{x_i})=\\
	=& g_{x_i}\bigg(\big(d_{x_i}(A^{-1}B) \big) V_{x_i}', \big(d_{x_i}(A^{-1}B) \big) W_{x_i}'\bigg)=\\
	=& \inner{Ad(A^{-1}B) V',Ad(A^{-1}B) W'}_\SL= \inner{V', W'}_\SL=\\
	=& (g_i)_{B\cdot x_i} ((d_{x_i}B) V'_{x_i}, (d_{x_i}B) W'_{x_i}).
	\end{align*}
	Since $\PSL$ acts on $X_i$ transitively and by diffeomorphisms, $g_i$ is uniquely defined.
	
	By construction $\PSL$ acts by isometries on both $g_1$ and $g_2$.  
	
\item The metric $(g_i)_{A\cdot x_i}$ is the push-forward via $\beta_i$ of $\inners_{|\mathcal D_i}$, i.e. \[d_A\beta_i\colon (D_i (A), \inners_\SL) \to (T_{A\cdot x_i} X_i, g_i)\] is a linear isometry for all $A$. Indeed, by equation $\eqref{eq 2}$, for all $V, W\in \mathfrak m_i$,
\begin{equation}
\begin{split}
&(g_i)_{A\cdot x_i} \big( (d_A\beta_i)(d_{I_2} L_A) (V),(d_A\beta_i)(d_{I_2} L_A) (W)\big)=\\
&=(g_i)_{A\cdot x_i} \big( (d_{x_i}A) (d_{I_2}\beta_i)(V),(d_{x_i}A) (d_{I_2}\beta_i) (W)\big) =\\
&= \inner{V, W}_\SL = \inner{(d_{I_2} L_A)(V), (d_{I_2} L_A)(W) }_\SL.
\end{split}
\end{equation}

\item Since $\beta_1$ is smooth and $\beta_2$ is holomorphic, one can easily see that $g_1$ is Riemannian and $g_2$ is holomorphic Riemannian.

\item We compute the sectional curvature of $g_i$.  Let $V, W\in \mathfrak m_i$ be orthonormal with respect to $\inners$, then by Proposition \ref{Prop PSL} and equation \eqref{eq curvatura spazio simmetrico} we have
\begin{align*}
&-1= K(Span_\C (V, W))= \inner{R^{\PSL}(V, W) W, V}= -\frac 1 4 \inner{  [[V,W], W], V}= \\
& =- \frac 1 4 \inner{R^{X_i} (V_{x_i}, W_{x_i})W_{x_i}, V_{x_i} }= \frac 1 4 K^{X_i} (Span_\C (V_{x_i},W_{x_i}) ).
\end{align*}
We conclude that $(X_i, g_i)$ has constant sectional curvature $-4$ and the proof follows.
\end{itemize}

\end{proof}

\end{document}